\documentclass[10pt,a4paper]{amsart}
\usepackage{amssymb, amscd, verbatim, amsthm}
\usepackage[arrow, matrix, curve]{xy}
\usepackage{hyperref}
\usepackage{amsmath}
\usepackage[a4paper]{geometry}
\usepackage{mathrsfs}
\usepackage{appendix}

\newcommand{\Fr}{\mathrm{Fr}}
\newcommand{\aind}{{\mathrm{ind}}}

\advance\textheight by 1.25cm

\title{A variational  open image theorem in positive characteristic}
\author{Gebhard B\"ockle, Wojciech Gajda  and Sebastian Petersen}

\newtheorem{leer}{}[section]
\newtheorem{thm}[leer]{Theorem}

\newtheorem{prop}[leer]{Proposition}
\newtheorem{lemm}[leer]{Lemma}
\newtheorem{coro}[leer]{Corollary}
\newtheorem{mthm}{Theorem}

\theoremstyle{remark}    
\newtheorem{rema}[leer]{Remark}

\newcommand{\DD}{{\mathscr{D}}}

\newcommand{\XX}{{\mathscr{X}}}

\newcommand{\elli}{{\ell^\infty}}

\newcommand{\uG}{{\underline{G}}}
\newcommand{\uD}{{\underline{D}}}

\newcommand{\cD}{{\underline{\DD}}}

\newcommand{\gen}{{\mathrm{gen}}}

\newcommand{\Lll}{{\mathbb{L}'}}

\DeclareMathOperator{\Gal}{Gal}
\DeclareMathOperator{\Spec}{Spec}
\newcommand{\GL}{\underline{\mathrm{GL}}}

\newcommand{\trdeg}{\mathrm{trdeg}}

\newcommand{\ol}[1]{\overline{#1}}
\newcommand{\ul}[1]{\underline{#1}}

\DeclareMathOperator{\chara}{char}

\newcommand{\scc}{\mathrm{sc}}

\newcommand{\sss}{\mathrm{ss}}

\def\Ff{{\mathbb{F}}}

\def\Qq{{\mathbb{Q}}}
\def\Zz{{\mathbb{Z}}}
\def\Aa{{\mathbb{A}}}
\def\Nn{{\mathbb{N}}}

\begin{document}

\parindent0em
\parskip.5em

\maketitle

\begin{abstract} In this note we prove a variational open adelic image theorem for the Galois action on the cohomology of smooth proper $S$-schemes where $S$ is a smooth variety over a finitely generated field of positive characteristic. A central tool is a recent result of Cadoret, Hui and Tamagawa.
\end{abstract}

\section*{Introduction}
Let $k$ be a finitely generated infinite field of characteristic $p>0$,
$S$ a smooth geometrically connected $k$-variety of positive dimension and $f: \mathscr{X}\to S$ a smooth proper morphism of 
schemes. Let $K=k(S)$ be the function field of $S$ and let
$X/k(S)$ be the generic fibre of $\mathscr{X}$. 
Fix $j\in\Nn$.  For every prime number $\ell\neq p$ we define
$V_\ell:=H^j(X_{\ol{K}}, \Qq_\ell)$ and let 
$$\rho_\elli: \pi_1(S)\to \GL_{V_\ell}(\Qq_\ell)$$
be the representation of $\pi_1(S)$ on the $\Qq_\ell$-vector space $V_\ell$. 
We write
$$\rho: \pi_1(S)\to \prod_{\ell\neq p}  \GL_{V_\ell}(\Qq_\ell)$$
for the induced adelic representation $\prod_{\ell\neq p}\rho_\elli$. 

For every
point $s\in S$ with residue field $k(s)$
we denote by $s_*: \Gal(k(s))\to \pi_1(S)$ the homomorphism induced by $s$
(well-defined up to conjugation), and  for any group homomorphism $\tau: \pi_1(S)\to H$,
we define $\tau_s:=\tau\circ s_*: \Gal(k(s))\to  H$ as the {\em specialization of $\rho$ at~$s$}.
Note that $\rho_{\elli, s}$ is isomorphic to the representation of $\Gal(k(s))$ on 
$H^j(X_{s, \ol{k(s)}}, \Qq_\ell)$ where $X_s=\mathscr{X}\times_S \Spec(k(s))$ is the 
special fibre of $\mathscr{X}$
in $s$, see~\cite[VI.~Cor.~4.2]{milne}. 

The aim of this paper is to study the variation of the monodromy groups
$\rho_{\elli, s}(\Gal(k(s)))$ (resp. $\rho_{s}(\Gal(k(s)))$) 
for closed points $s\in S$ in comparing them to the
corresponding monodromy group $\rho_{\elli}(\pi_1(S))$ (resp. $\rho(\pi_1(S))$) of the 
generic point of $S$.

For every prime number $\ell\neq p$ let $\uG(\rho_{\elli, s})$ (resp.
$\uG(\rho_{\elli})$)
 be the connected component of the Zariski closure of
$\rho_{\elli, s}(\Gal(k(s)))$ (resp. of
$\rho_{\elli}(\pi_1(S))$)
in $\GL_{V_\ell}/\Qq_\ell$, and define
$$S^\gen(\rho_{\elli})= \{s\in S\ \mbox{a closed point}: \uG(\rho_{\elli, s})=\uG(\rho_{\elli})\}$$
Being in $S^\gen(\rho_\elli)$ is a priori weaker than being $\ell$-Galois generic in the sense of Cadoret-Kret (see~\cite[3.1]{cadoret-kret}, \cite[1.5.3]{cadoretadelic} and \cite[\S~6]{pink}). By Theorem~\ref{thm-main-intro}(c), however, the notions are equivalent.

The following result is the main theorem of the present work.
\begin{mthm} (see Proposition \ref{indep}, Lemma~\ref{Lem-genbig}, Theorem~\ref{mainthm})
\label{thm-main-intro}
\begin{enumerate}
\item[(a)] The sets $S^\gen(\rho_{\elli})$ are independent of $\ell$.\! Let $S^\gen(\XX/S):=S^\gen(\rho_{\elli})$ for any $\ell\neq p$.
\item[(b)] The set $S^\gen(\XX/S)$ is Zariski dense in $S$, and in particular it is infinite.
\item[(c)] The group $\rho_s(\Gal(k(s)))$ is open in $\rho (\pi_1(S))$ for every $s\in 
S^\gen(\XX/S)$. 
\end{enumerate}
\end{mthm}

The above result relies on a similar result that holds if one replaces $S$ by its base change $S_{\ol{\Ff}_p k}$ under $k\to \ol\Ff_p k$. This base change allows us to apply standard tools to derive (a) and (b), and recent results from \cite{ctss} by Cadoret, Hui and Tamagawa to deduce (c). Given these results, the proof of Theorem~\ref{thm-main-intro} is rather elementary.

In the case where $k$ is a finitely generated field of zero characteristic, Cadoret established a theorem analogous to our main theorem in the case where $\XX/S$ is an abelian scheme (cf.~\cite{cadoretadelic}). Her theorem offers a strong tool to reduce the proof of conjectures about Galois representations attached to abelian varieties over finitely generated fields of zero characteristic to the number field case. Similarly our theorem can be used in order to reduce the proof of 
conjectures about smooth proper varieties over finitely generated fields of positive characteristic to the 
case where the ground field is a global function field. 

The results from \cite{ctss} also give one a rather precise conceptual description of $\rho(\pi_1(S_{\ol{\Ff}_pk}))$, as we shall explain in Section~\ref{LHSection}, and as is presumably well-known to the authors of \cite{ctss}: Let $\ul{D}_\elli$ denote the connected component of the Zariski closure of $\rho_{\elli}(\pi_1(S_{\ol{\Ff}_pk}))$, or, equivalently, the derived group of $\uG(\rho_{\elli})$ (see~Corollary~\ref{cor-derived}). In Theorem~\ref{LHTheorem} we prove that for $s\in S^\gen(\XX/S)$ the group $\rho_s(\Gal(\overline{\Ff}_pk(s)))$ generates a special adelic subgroup in $\prod_{\ell\neq p} \ul{D}^{\mathrm{sc}}_\elli(\Qq_\ell)$ in the sense of Hui and Larsen \cite{huilarsen}, where $ \ul{D}^{\mathrm{sc}}_\elli\to  \ul{D}_\elli$ denotes the simply connected cover. 

After submitting our manuscript to ArXiv we were informed by Anna Cadoret that she has a manuscript \cite{cadoretadelic2} with similar results, now also available on her homepage.

{\bf Acknowledgements} 

Part of this work was done during a stay of the three authors at Adam Mickiewicz university in Pozna\'n financed by NCN grant no. UMO-2014/15/B/ST1/00128. G.B. was supported by the DFG within the SPP1489 and the FG1920. 
A previous version of this work addressed only the case of abelian varieties over a base variety in positive characteristic. We thank the referee of that version for much valuable feedback that is also relevant for the present generalization.

\section*{Notation}
For a field $k$ we denote by $\ol{k}$ an algebraic closure and by $\Gal(k)$ the absolute Galois group of $k$. For a $k$-variety $S$ we denote by $k(S)$ its function field, by $|S|$ its set of closed points, and by $\pi_1(S)$ the \'etale fundamental group of $S$ with base point the geometric generic point $\Spec(\ol{k(S)})\to S$. If $\chara k=p$, let $S_{\ol\Ff_p k}$ denote the base change of $S$ under $k\to \ol\Ff_p k$.

Suppose that $V$ is a finite-dimensional $\Qq_\ell$-vector space, $\Pi$ is a profinite group
and $\rho: \Pi\to  \GL_{V}(\Qq_\ell)$ is a continuous homomorphism. We denote by $\uG(\rho)$ the connected component of the Zariski closure of
$\rho(\Pi)$ in $\GL_{V}$. Then $\uG(\rho)$ is an algebraic group over $\Qq_\ell$. It is reductive if $\rho$ is semisimple (see~\cite[22.138]{milne-iag}). We write $\rho|H$ for the restriction of $\rho$ to a subgroup $H$ of $\Pi$, and we denote by $\Pi^+$ the closed subgroup of $\Pi$ generated by its pro-$\ell$ Sylow subgroups.

If $\Pi=\pi_1(S)$, then we define the {\em set of Galois generic points with respect to~$\rho$} as
$$S^\gen(\rho):=\{s\in |S|: \uG(\rho)=\uG(\rho_s)\}.$$
 By $\Lll$ we denote the set of all prime numbers $\ell\neq p$.
 For a linear algebraic group $\ul{G}$ defined over a field, we denote by $\cD\ul{G}$ its derived group.

Let $K$ be a finitely generated field of characteristic $p> 0$.
We call a family $(\rho_\elli: \Gal(K)\to \GL_{V_\ell}(\Qq_\ell))_{\ell\in\Lll}$ of continuous homomorphisms, where the 
$V_\ell$ are finite-dimensional $\Qq_\ell$-vector spaces, a {\em strictly compatible system over $K$ pure of weight $j$} if there exists a smooth $\Ff_p$-variety $T$ with $\Ff_p(T)=K$ such that the following properties (i) and (ii) hold: 
\begin{itemize}
\item[(i)] $\rho_\elli$ factors through $\pi_1(T)$ for every $\ell\in\Lll$.
\item[(ii)] For every $t\in |T|$, denoting by $\Fr_t\in\Gal(\Ff_p(t))$ the arithmetic Frobenius $x\mapsto x^{\frac{1}{|\Ff_p(t)|}}$, the characteristic polynomial of $\rho_{\elli, t}(\Fr_t)$ has coefficients in $\Zz$, it is independent of $\ell$, and its roots all have absolute value $|\Ff_p(t)|^{\frac{j}{2}}$.
\end{itemize}

\section{Preliminaries}\label{PrelimSec}
In this section we collect basic results, mostly not due to the present authors, for use in later sections.
Let $K$ be a finitely generated infinite field of characteristic $p>0$ and $X/K$ a smooth proper scheme. 
Fix $j\in\Nn$ and define for $\ell\in\Lll$ 
$$\begin{array}{rcl}
V_\ell(X)&:=&H^j(X_{\ol{K}}, \Qq_\ell)\ \mbox{and}\\
T_\ell(X)&:=&H^j(X_{\ol{K}}, \Zz_\ell)/(\mathrm{Torsion}).\\
\end{array}$$
Then $V_\ell(X)$ is a finitely generated $\Qq_\ell$-vector space, $T_\ell(X)$ is a finitely generated free 
$\Zz_\ell$-module and $V_\ell(X)=T_\ell(X)\otimes \Qq_\ell$ for all $\ell\in\Lll$. Let $\rho_\elli$ be the 
representation of $\Gal(K)$ on $V_\ell(X)$. 

By Deligne's theorem on the Weil conjectures (cf.~\cite[Thm.~1.6]{deligne-weil1}), standard spreading-out principles, and proper-smooth base change (cf.~\cite[VI.~Cor.~4.2]{milne}), one has the following result.

\begin{thm}(Deligne)
\label{Thm-CompatSys}
The family of representations $(\rho_\elli)_{\ell\in\Lll}$ is a strictly compatible system over $K$ pure of weight $j$.
\end{thm}

The following combines results from  \cite{bgp}, \cite{ct}, \cite{lp}  and \cite{serretoribet1}.
\begin{thm} 
\label{Thm-connext} There exists a finite Galois extension $K_\aind/K$
with the following properties. 
\begin{enumerate}
\item[(a)] For all $\ell\in\Lll$, one has 
\[\hbox{$\rho_{\elli}(\Gal(K_\aind))\subset \uG(\rho_\elli)(\Qq_\ell)$ and $\rho_{\elli}(\Gal(\ol{\Ff_p} K_\aind))\subset \uG(\rho_\elli|\Gal(\ol{\Ff_p} K_\aind) )(\Qq_\ell)$.}\]
\item[(b)] One has $\rho_\elli(\Gal(\ol{\Ff}_p K_\aind))=\rho_\elli(\Gal(\ol{\Ff}_p K))^+$ for all $\ell\gg0$ in $\Lll$.
\item[(c)] If $\rho: \Gal(K)\to \prod_{\ell\in\Lll} \GL_{T_\ell(X)} (\Zz_\ell)$
is the homomorphism induced by $\prod_{\ell\in\Lll}\rho_{\elli}$, then 
$$\rho(\Gal(\ol{\Ff}_p K_\aind))=\prod_{\ell\in\Lll} \rho_{\elli} (\Gal(\ol{\Ff}_p K_\aind)).$$
\item[(d)] If $K=k(S)$ for 
a smooth
 $k$-variety $S$ such hat $\rho$ factors via $\pi_1(S)$, then one can further require that there exists a connected finite \'etale cover $S_\aind$ of $S$ such that $K_\aind=k(S_\aind)$.
\end{enumerate}
\end{thm}
\begin{proof}
The existence of a finite Galois extension $K'/K$ such that the first half of (a) holds is due to Serre (see~\cite[2nd letter]{serretoribet1}); see also \cite[Prop.~6.14]{lp1}. For the second half of (a) this follows from \cite[2.2]{lp} by Larsen and Pink. It follows from \cite[Thm.~7.7]{bgp}, or alternatively from \cite{ct}, that after replacing $K'$ 
by a larger finite Galois extension $K_\aind$ of $K$ also (b) and (c) hold.

Let $S_\aind$ be the connected finite \'etale cover of $S$ corresponding to the quotient 
$\rho(\Gal(K_\aind))$ of $\pi_1(S)$. Then (a)--(d) hold if 
we replace $K_\aind$ by $k(S_\aind)$.
\end{proof}

\begin{thm}(Deligne, Grothendieck, Serre)
\label{Thm-serrebrisgen}
\begin{enumerate}
\item[(a)] The restriction $\rho_\elli|\Gal(\ol{\Ff}_p K)$ is semisimple.
\item[(b)] The group $\uG(\rho_\elli|\Gal(\ol{\Ff}_p K))$ is semisimple.
\item[(c)] The group $\rho_\elli(\Gal(\ol{\Ff}_p K_\aind))$ is open in $\uG(\rho_\elli|\Gal(\ol{\Ff}_p K))$. 
\end{enumerate}
\end{thm}

\begin{proof} Part (a) is \cite[Cor. 3.4.13]{deligne-weil2} and due to Deligne. Part (b) is attributed by Deligne to Grothendieck and given in \cite[Cor.~1.3.9]{deligne-weil2}. Part (c) is due to Serre (cf.~\cite{serre67}).
\end{proof}

\begin{coro}\label{cor-derived}
We have $\cD\uG(\rho_\elli)=\uG(\rho_\elli|\Gal(\ol{\Ff}_p K))$.
\end{coro}
\begin{proof}
Clearly $\rho_\elli(\Gal(K))$ normalizes $\rho_\elli(\Gal(\ol{\Ff}_p K))$. This is preserved under closures, so that $\uG(\rho_\elli|\Gal(\ol{\Ff}_p K))$ is a normal subgroup of $\uG(\rho_\elli)$.
In particular the quotient $\ul{Q}:=\uG(\rho_\elli)/\uG(\rho_\elli|\Gal(\ol{\Ff}_p K))$ is a connected linear algebraic group. Now $\ul Q$ contains as a Zariski dense subset a finite index subgroup of $\rho_\elli(\Gal(K))/\rho_\elli(\Gal(\ol{\Ff}_p K))$, and the latter is a quotient of $ \Gal(\ol\Ff_pK/K)\cong\hat\Zz$ and thus $\ul Q$ is abelian. From the universal property of the derived group, we deduce $\cD\uG(\rho_\elli)\subset\uG(\rho_\elli|\Gal(\ol{\Ff}_p K))$. By Theorem~\ref{Thm-serrebrisgen}(b), we have $\cD\uG(\rho_\elli|\Gal(\ol{\Ff}_p K))=\uG(\rho_\elli|\Gal(\ol{\Ff}_p K))$, and so also $\uG(\rho_\elli|\Gal(\ol{\Ff}_p K))\subset \cD\uG(\rho_\elli)$.
\end{proof}

\section{The sets of Galois generic points}
In addition to the data introduced above let  $k$ be an {\bf infinite}
 finitely generated field and $S/k$  a smooth geometrically connected variety with function field $K$. Assume that $X$ extends to a smooth proper scheme $\XX$ over $S$. Then $\rho_\elli$ factors through  $\pi_1(S)$. As recalled in the introduction, for every  $s\in |S|$ 
the representation $\rho_{\elli, s}$ is isomorphic to the representation of $\Gal(k(s))$ on $H^j(X_{s, \ol{k(s)}},
\Qq_\ell)$ where $X_s=\XX\times_S \Spec(k(s))$ is the special fibre of $\XX$ in~$s$. For our applications below, note that the results of Section~\ref{PrelimSec} also apply to $((\rho_{\elli,s})_{\ell\in\Lll},k(s))$ instead of $((\rho_{\elli})_{\ell\in\Lll},K)$.
%

In this section we group together various results about the sets $S^\gen(\rho_\elli)$ and some consequences.  The following result is due to Serre (cf.~\cite[\S~6.10]{Serre-MW}). We outline the argument.
\begin{lemm}\label{Lem-genbig} 
For any $\ell\in\Lll$ the set $S^\gen(\rho_\elli)$ is Zariski-dense in $S$.
\end{lemm}

{\em Proof.}
Let $\ell$ be in $\Lll$ and let $\Phi_\ell$ be the Frattini subgroup of $G:=\rho_\elli(\pi_1(S))$, i.e., the intersection of all maximal closed subgroups of $G$. Clearly $G$ is a compact subgroup of $\GL_{V_\ell}(\Qq_\ell)$, and so by \cite[Thm. 8.33., p. 201]{dixon} there is an open pro-$\ell$ subgroup of $G$ of finite rank. We deduce from \cite[\S~6.10~Prop.]{Serre-MW} that $\Phi_\ell$ is open in~$G$.

Consider now the composite homomorphism:
$$\ol{\rho}_\ell\colon \Gal(K) \buildrel\rho_\elli\over\longrightarrow G \to  G/\Phi_\ell.$$
Both $\rho_\elli$ and $\ol\rho_\ell$ factor via $\pi_1(S)$, and by the universal property of the Frattini group, we have  
\begin{equation}\label{Eqn-M}
 \{s\in |S|: \rho_{\elli, s}(\Gal(k(s)))=G\}=\{s\in |S|: \ol{\rho}_{\ell, s}(\Gal(k(s)))=G/\Phi_\ell\}.
\end{equation}
Let $M$ denote the right hand set. Because $G/\Phi$ is finite and $k$ is Hilbertian, the set $M$ is~Zariski dense in $S$. This completes the proof, because $S^\gen(\rho_\elli)$ contains  the left hand side of (\ref{Eqn-M}).
\hfill $\Box$

For the rest of this paper we define, for $s\in |S|$ and $\ell\in\Lll$, the semisimple groups
$$\uD_\elli:=\uG(\rho_\elli|\Gal(\ol{\Ff}_p K))\ \mbox{and}\  
\uD_{\elli, s}:=\uG(\rho_{\elli, s}|\Gal(\ol{\Ff}_p k(s)))$$
over $\Qq_\ell$ and recall that $\uD_\elli=\cD\uG(\rho_\elli)$ and $\uD_{\elli, s}=\cD\uG(\rho_{\elli, s})$.

Because of Theorem \ref{Thm-CompatSys}, the following result follows from  \cite[Thm.~2.4]{lp}.
\begin{thm} (Larsen and Pink) \label{thm-cprop} 
If $\trdeg_{\Ff_p}K=1$, then the functions $\ell\mapsto \dim\uD_\elli$ and $\ell\mapsto \dim\uD_{\elli,s}$ on $\Lll$ are both constant.
\end{thm}
As an application of Lemma~\ref{Lem-genbig} we extend Theorem~\ref{thm-cprop} to all fields $K$ considered here.
\begin{coro}\label{coro-cprop}
The functions $\ell\mapsto \dim\uD_\elli$ and $\ell\mapsto \dim\uD_{\elli,s}$ on $\Lll$ 
are both constant.
\end{coro}
{\em Proof.}
Note that it suffices to prove the assertion on $\ell\mapsto \dim\uD_\elli$, since for the second assertion one may take $k(s)$ for $K$ and $\rho_{\elli,s}$ for $\rho_\elli$. Let now $K$ be a finitely generated field over $\Ff_p$. We choose a subfield $\kappa\subset K$ such that $\trdeg_{\Ff_p}\kappa=1$ and $K/\kappa$ is a regular extension of fields. 
Next we choose a geometrically connected smooth $\kappa$-variety $B$ with $\kappa(B)=K$ and a smooth proper morphism $\XX_B\to B$ with generic fibre $X/K$. Let $\ell_0\in\Lll$ be such that $\dim(\uD_{\ell_0^\infty})=\max_{\ell}\dim(\uD_\elli) $. By Lemma~\ref{Lem-genbig} there exists a point $b\in B^\gen(\rho_{\ell_0^\infty})$; note that $\trdeg_{\Ff_p}\kappa(b)=1$. Then for any $\ell\in\Lll$ we have
\[\dim(\uD_\elli)\ge \dim(\cD\uG(\rho_{\elli, b})) \stackrel{\mathrm{Thm.\,}\ref{thm-cprop}}=\dim(\cD\uG(\rho_{\ell_0^\infty, b}))\stackrel{\mathrm{choice\,of\,}b}= \dim(\uD_{\ell_0^\infty}),\]
and it follows from the choice of $\ell_0$ that $\ell\mapsto\dim\uD_\elli$ is constant.
\hfill $\Box$
\begin{rema}
In the above proof, for the reduction from $K$ to transcendence degree $1$, one could also use results on ``space filling curves'', as for instance \cite[Rem.~2.18(ii)]{Drinfeld}, cf.~\cite[Ex.~3.1]{ct}.
\end{rema}

We also need an analog of \cite[3.2.3]{cadoret-kret} for $S^\gen$, as defined here.
\begin{lemm}
\label{Lem-GalGen}
\begin{enumerate}
\item[(a)] For $\bar s\in |S_{\ol\Ff_p k}|$ denote by $\bar s_S$ the closed point of $S$ under $\bar s$. Then
\[S^\gen(\rho_\elli) = \{\bar s_S\mid \bar s\in S_{\ol\Ff_p k}^\gen(\rho_\elli|\pi_1(S_{\ol\Ff_p k})) \}.\]
\item[(b)] If $S'$ is a finite \'etale cover of $S$ and for $s'\in|S'|$ denote by $s'_S$ the closed point of $S$ under $s'$. Then $S^\gen(\rho_\elli) = \{s'_S\mid s'\in (S')^\gen(\rho_\elli |\pi_1(S')) \}$.
\end{enumerate}
\end{lemm}
\begin{proof} We only prove (a), the proof of (b) being elementary. There is a bijection between points in $|S|$ and orbits under $\Gal(\ol\Ff_p k/k)$ in $|S_{\ol\Ff_p k}|$. So let $s$ be in $|S|$ and denote by $\bar s$ a point in $|S_{\ol\Ff_p k}|$ above it. Consider the commutative diagram
\[\xymatrix{
\uD_{\elli,s} \ar@{^{(}->}[r]\ar@{^{(}->}[d]& \uG(\rho_{\elli, s})\ar@{^{(}->}[d]\\
\uD_\elli  \ar@{^{(}->}[r]& \uG(\rho_\elli)\rlap{.}\\
}\]
If $s$ is Galois generic, then the right vertical inclusion is an isomorphism. Hence by Corollary~\ref{cor-derived} the same holds for the left inclusion, and this means that $\bar s$ is Galois generic. Conversely, let $\bar s$ be Galois generic, so that the left vertical inclusion is an isomorphism, and we have an induced monomorphism 
$$\iota_s\colon \uG(\rho_{\elli, s})/ \uD_{\elli,s}\hookrightarrow \uG(\rho_\elli)/ \uD_\elli $$
of algebraic groups (see the proof of Corollary~\ref{cor-derived}).
  Now there exists a finite Galois extension $k'/k(s)$ such that the image of the group $\Gal(\ol{\Ff}_pk'/k')$ is Zariski dense in the connected group $\uG(\rho_{\elli, s})/ \uD_{\elli,s}$. Because $\Gal(\ol{\Ff}_pk(s)/k(s))\hookrightarrow\Gal(\ol{\Ff}_pk/k)$ is of finite index, the group $\Gal(\ol{\Ff}_pk'/k')$ is also Zariski dense in the connected group $\uG(\rho_\elli)/ \uD_\elli$. Hence $\iota_s$ is an isomorphism, and it follows that $s$ is Galois generic.
\end{proof}

\begin{rema}
 Let $\ell$ be in $\Lll$ and denote by $\rho_\elli^\sss$ the semisimplification of $\rho_\elli$. Then from Lemma~\ref{Lem-GalGen}, Theorem~\ref{Thm-serrebrisgen} and Corollary~\ref{cor-derived} it is immediate that $S^\gen(\rho_\elli^\sss)=S^\gen(\rho_\elli)$. 
\end{rema}

\begin{prop} \label{indep} For any two primes  $\ell_1, \ell_2\in\Lll$ we have $S^\gen(\rho_{\ell_1^\infty})=
S^\gen(\rho_{\ell_2^\infty})$. 
\end{prop}

\begin{proof}
By Lemma~\ref{Lem-GalGen} it suffices to show 
\begin{equation}\label{Eqn-GalGen}
S_{\ol\Ff_p k}^\gen(\rho_{\ell_1^\infty}|\Gal(\ol\Ff_p) K)=S_{\ol\Ff_p k}^\gen(\rho_{\ell_2^\infty}|\Gal(\ol\Ff_p K)). 
\end{equation}
For this let $\bar s$ be in $| S_{\ol\Ff_p k}|$. Observe that for any $\ell\in\Lll$ we have the obvious 
assertion~(a${}_\ell$) that $\uD_{\elli, s} \hookrightarrow \uD_\elli$
is an inclusion of connected semisimple groups and from Corollary~\ref{coro-cprop} the assertion (b) 
that both functions
$\ell\mapsto \dim (\uD_\elli)$ and $\ell\mapsto \dim (\uD_{\elli, s})$ on $\Lll$ are constant.
From these one deduces the following chain of equivalences
\begin{eqnarray*}
\bar s\in S_{\ol\Ff_p k}^\gen(\rho_{\ell_1^\infty}|\Gal(\ol\Ff_pK)) &\stackrel{(a_{\ell_1})}\Longleftrightarrow& 
\dim \uD_{\ell^\infty_1, s}=\dim \uD_{\ell^\infty_1} \\
&\stackrel{(b)}\Longleftrightarrow& 
\dim \uD_{\ell^\infty_2, s}=\dim \uD_{\ell^\infty_2}     \\
&\stackrel{(a_{\ell_2})}\Longleftrightarrow& \bar s\in S_{\ol\Ff_p k}^\gen(\rho_{\ell_2^\infty}|\Gal(\ol\Ff_pK)).
\end{eqnarray*} 
\end{proof}

We define $S^\gen(\XX/S)=S^\gen(\rho_\elli)$ for any $\ell\in\Lll$, and we call $S^\gen( \XX/S)$ the {\em set of Galois generic points 
of $\XX/S$}. 

\section{Adelic openness}
\label{Section-AdelicOpen}
{\em Throughout this section, we shall assume $K=K_\aind$, cf.~Theorem~\ref{Thm-connext}, since the main result of this section is the proof of Theorem~\ref{thm-main-intro}(c) which may be proved over $K_\aind$ by Lemma~\ref{Lem-GalGen}(b).}
\begin{lemm} \label{Lem-serre} 
Let $s$ be in $S^\gen(\XX/S)$. Then for all $\ell\in\Lll$ the group $\rho_{\elli, s}(\Gal(\ol{\Ff}_p k(s)))$ is open in $\rho_{\elli}(\Gal(\ol{\Ff}_p K))$.
\end{lemm} 

\begin{proof}
Because $K=K_\aind$, the group $\rho_{\elli}(\Gal(\ol{\Ff}_p K))$ lies in $\uD_{\elli}(\Qq_\ell)$, and hence so does $\rho_{\elli, s}(\Gal(\ol{\Ff}_p k(s)))\subset \rho_{\elli}(\Gal(\ol{\Ff}_p K))$. By our choice of $s$ we have $\uD_{\elli}=\uD_{\elli,s}$, and thus by Theorem \ref{Thm-serrebrisgen}(c), both $\rho_{\elli}(\Gal(\ol{\Ff}_p K))$  and $\rho_{\elli, s}(\Gal(\ol{\Ff}_p k(s)))$ are open in $\uD_{\elli}(\Qq_\ell)$. This implies the asserted openness and completes the proof.
\end{proof}

Note that $\rho_\elli$ has its image in $\GL_{T_\ell(X)}(\Zz_\ell)$. 
Let $\cD_\elli/\Zz_\ell$ be the Zariski closure of $\uD_\elli$ in $\GL_{T_\ell(X)}$. 
The following result is powered by two theorems from a recent paper of Cadoret, Hui and Tamagawa (cf.~\cite{ctss}).

\begin{thm} 
\label{thm-open1}
For all $\ell\gg 0$ we have:
\begin{enumerate}
\item[(a)] The group scheme $\cD_\elli/\Zz_\ell$ is semisimple.
\item[(b)] If $s\in S^\gen(\XX/S)$, then we have
$$\rho_{\elli, s}(\Gal(\ol{\Ff}_p k(s)))^+=\rho_{\elli}(\Gal(\ol{\Ff}_p K))=\cD_\elli(\Zz_\ell)^+.$$
\end{enumerate}
\end{thm}

\begin{proof} Part (a) is immediate from  \cite[Thm. 1.2]{ctss} and \cite[Cor. 7.5]{ctss}. For part (b), let $s\in S^\gen(\XX/S)$, so that $\uD_{\elli, s}=\uD_{\elli}$. 
Then $\cD_{\elli}$ is also the Zariski closure of $\rho_{\elli, s}(\Gal(\ol{\Ff}_p k(s))$ in $\GL_{T_\ell(X)}$. We now apply \cite[7.3]{ctss} twice in order to get
$$\begin{array}{rcl}
\cD_\elli(\Zz_\ell)^+&=&\rho_\elli(\Gal(\ol{\Ff}_p K))^+ \ = \ \rho_{\ell^\infty}(\Gal(\ol{\Ff}_p K))\ \mbox{ and}\\
\cD_\elli(\Zz_\ell)^+&=&\rho_\elli(\Gal(\ol{\Ff}_p k(s)))^+.
\end{array}
$$
\end{proof}

\begin{coro}\label{cor-for-main}
Consider the adelic representation
$$\rho: \Gal(K)\to \prod_{\ell\in\Lll} \GL_{T_\ell(A)} (\Zz_\ell).$$
For every $s\in S^\gen(\XX/S)$ the group $\rho_s(\Gal(\ol{\Ff}_p k(s)))$ is open in $\rho (\Gal(\ol{\Ff}_p K))$.
\end{coro} 

{\em Proof.} Let $G=\Gal(\ol{\Ff_p} K)$ and $G_s=\Gal(\ol{\Ff_p} k(s))$. Then 
$$\rho(G)=\prod_{\ell\in\Lll} \rho_{\elli}(G)$$
by Theorem~\ref{Thm-connext} because $K=K_\aind$. Furthermore, again by Theorem~\ref{Thm-connext}, there exists an open normal subgroup $H_s=\Gal(\overline\Ff_pk(s)_\aind)$ of $G_s$ such that 
$$\rho_s(H_s)=\prod_{\ell\in \Lll} \rho_{\elli, s} (H_s).$$
For every prime number $\ell\in \Lll$ the group $\rho_{\elli, s} (H_s)$ is open in
$\rho_{\elli, s} (G_s)$ (because $H_s$ is open in $G_s$), and $\rho_{\elli, s} (G_s)$ is open in 
$\rho_\elli(G)$ by Lemma~\ref{Lem-serre}. It follows that $\rho_{\elli, s} (H_s)$ is open in 
$\rho_\elli(G)$ for all $\ell\in\Lll$. 

By Theorem~\ref{thm-open1} we have $\rho_{\elli, s}(G_s)^+=\rho_{\elli}(G)$ for all $\ell\gg 0$ in $\Lll$ because $K=K_\aind$. From our construction of $H_s$ via  Theorem \ref{Thm-connext}, we deduce $\rho_{\elli, s}(H_s)=\rho_{\elli, s}(G_s)^+$ for all $\ell\gg0$ in $\Lll$, and hence $\rho_{\elli, s} (H_s)=\rho_\elli(G)$ for these $\ell$. By the definition of the product topology, the group $\rho_s(H_s)$ is open in $\rho(G)$. As $\rho_s(H_s)\subset \rho_s(G_s)\subset \rho(G)$, the assertion follows.\hfill $\Box$

\begin{thm} 
For $s\in S^\gen(\XX/S)$ the group $\rho_s(\Gal(k(s)))$ is open in $\rho (\Gal(K))$.
\label{mainthm} 
\end{thm}

\begin{proof}
Let $\bar s\in |S_{\ol\Ff_pk}|$ be above $s\in S^\gen(\XX/S)$, and consider the following commutative diagram with exact rows, where $\ol\rho$ is induced from $\rho$:
$$\xymatrix{1\ar[r] & \rho(\pi_1(S_{\ol\Ff_p k})) \ar[r] & \rho(\pi_1(S)) \ar[r] & \rho(\pi_1(S))/\rho(\pi_1(S_{\ol\Ff_p k})) \ar[r] & 1 \\
1\ar[r] & \pi_1(S_{\ol\Ff_p k}) \ar[r]\ar[u]_\rho & \pi_1(S)\ar[r]\ar[u]_\rho &  \Gal(\ol{\Ff_p} k/k) \ar[u]_{\ol\rho}\ar[r] & 1 \\
1\ar[r] & \Gal(\ol{\Ff_p} k(s)) \ar[r]\ar[u]_{\bar s_*} &\Gal(k(s)) \ar[r]\ar[u]_{s_*} &  \Gal(\ol{\Ff_p} k(s) /k(s)) \ar[u]_i\ar[r] & 1 
}$$
Now $\rho_s(\Gal(\ol{\Ff_p} k(s)))$ is open in $\rho(\pi_1(S_{\ol\Ff_p k}))$ by Corollary~\ref{cor-for-main}. Furthermore $\ol\rho(\Gal(\ol{\Ff_p} k(s) /k(s)) )$ is open in $\rho(\pi_1(S))/\rho(\pi_1(S_{\ol\Ff_p k}))$, because $k(s)/k$ is finite and $\ol{\rho}$ is surjective. It follows that $\rho_s(\Gal(k(s)))$ is open in $\rho (\Gal(K))$.
\end{proof}

\section{Largeness in the sense of Hui-Larsen} \label{LHSection}
Throughout this section we assume $K=K_\aind$ and fix a Galois generic point $s\in S^\gen(\XX/S)$. 
Let $p_\elli: \uD_\elli^\scc \to \uD_\elli$ be the simply connected cover of the semisimple $\Qq_\ell$-group
$\uD_\elli$.
Consider the groups
$$\begin{array}{rcl}
\uD_\Aa&:=&\prod\limits_{\ell\in\Lll} \uD_\elli(\Qq_\ell),\\
\uD_\Aa^\scc&:=&\prod\limits_{\ell\in\Lll} \uD_\elli^\scc(\Qq_\ell),\\
\Gamma&:=&\rho(\Gal(\ol{\Ff}_p K)\subset \uD_\Aa, \\
\Gamma_s&:=&\rho_s(\Gal(\ol{\Ff}_p k(s)))\subset \uD_\Aa.\end{array}
$$ 
The aim of this section is to prove that $\Gamma_s$ is a subgroup of 
$\uD_\Aa$ which is ``large'' in the sense of Hui and Larsen (cf.~\cite[Section 2]{huilarsen}). 
This shows that an 
analogue of conjecture \cite[Conj. 1.3]{huilarsen} for abelian varieties in positive characteristic holds.
Denote by $\kappa_\elli: \uD_\elli\times \uD_\elli\to \uD_\elli^{\scc}$ the commutator map and by 
$\kappa: \uD_\Aa\times \uD_\Aa\to \uD_\Aa^\scc$ the map derived from the $\kappa_\elli$. 

\begin{thm} \label{LHTheorem} $\kappa(\Gamma_s, \Gamma_s)$ (resp. $\kappa(\Gamma, \Gamma)$) generates a special adelic 
(cf.~\cite[Section 2]{huilarsen})  subgroup of $\uD_\Aa^\scc$ which is equal to $\kappa(\Gamma_s, \Gamma_s)^u$ 
(resp. to $\kappa(\Gamma, \Gamma)^u$) for some $u$. Moreover 
$\kappa(\Gamma_s, \Gamma_s)^2$  (resp. $\kappa(\Gamma, \Gamma)^2$) contains a special adelic subgroup of $\uD_\Aa^\scc$. 
\end{thm} 

{\em Proof.} Denote by $\mathrm{pr}_\elli: \uD_\Aa\to \uD_\elli$ the projection on the $\ell$-th factor of the product. Note that $\mathrm{pr}_\elli(\Gamma_s)=\rho_{\elli, s}(\Gal(\ol{\Ff}_p k(s)))$ is  Zariski dense in $\uD_\elli$ 
for each $\ell\in\Lll$ because $s$ is Galois generic. 
By \cite[Thm. 1.2]{bgp} there exists a finite extension $F/\overline{\Ff}_p k(s)$ such that 
$$\rho_s(\Gal(F))=\prod_{\ell\in\Lll} \rho_{\elli, s}(\Gal(F)).$$
Thus, if we define $\Gamma_{\elli, s}:=  \rho_{\elli, s}(\Gal(F))$, then $\prod_{\ell\in\Lll} \Gamma_{\elli, s}\subset
\Gamma_s\subset \Gamma$. For each $\ell\in \Lll$ the group $\Gamma_{\elli, s}$ is open in $\uD_\elli(\Qq_\ell)$ by 
Lemma~\ref{Lem-serre}. Furthermore, as $\cD_\elli$ is semisimple (cf.~Theorem~\ref{thm-open1}) for $\ell\gg 0$, 
\cite[Cor. 8.2]{ctss} implies that the group $\mathrm{pr}_\elli^{-1}(\Gamma_{\elli, s})$ is a hyperspecial 
maximal compact subgroup of $D_\elli^{sc}(\Qq_\ell)$ for $\ell\gg 0$. Thus the assertion follows by \cite[Thm. 3.8]{huilarsen}. \hfill $\Box$

{\small {\sc Gebhard B{\" o}ckle\\
Computational Arithmetic Geometry\\
IWR (Interdisciplinary Center for Scientific Computing)\\
University of Heidelberg\\
Im Neuenheimer Feld 368\\
69120 Heidelberg, Germany}\\
E-mail address: \texttt{\small gebhard.boeckle@iwr.uni-heidelberg.de} 
\par\medskip

{\sc Wojciech Gajda\\
Faculty of Mathematics and Computer Science\\
Adam Mickiewicz University\\
Umultowska 87\\
61614 Pozna\'{n}, Poland}\\
E-mail adress: \texttt{\small gajda@amu.edu.pl}
\par\medskip

{\sc Sebastian Petersen\\ 
Universit\"at Kassel\\
Fachbereich 10\\
Wilhelmsh\"oher Allee 73\\
34121 Kassel, Germany}\\
E-mail address: \texttt{\small petersen@mathematik.uni-kassel.de}}

\begin{thebibliography}{10}

\bibitem{bgp}
Gebhard B\"ockle, Wojciech Gajda and Sebastian Petersen.
\newblock {Independence of $\ell$-adic representations of geometric Galois groups}.
\newblock{J.\ Reine Angew.\ Math., DOI: 10.1515/crelle-2015-0024 (2015).}


\bibitem{cadoretadelic}
Anna Cadoret.
\newblock {An open adelic image theorem for abelian schemes}.
\newblock {Int.\ Math.\ Res.\ Not.\ Vol.\ 20, 10208--10242 (2015)}.

\bibitem{cadoretadelic2}
Anna Cadoret.
\newblock {An open adelic image theorem for motivic representations over function fields}.
\newblock {Preprint, submitted. Available at \textsc{http://www.cmls.polytechnique.fr/perso/cadoret/lIndepFF.pdf}}.

\bibitem{ctss}
Anna Cadoret, Chun-Yin Hui and Akio Tamagawa.
\newblock {Geometric monodromy - semisimplicity and maximality}.
\newblock {Preprint. August 2016.}


\bibitem{cadoret-kret}
Anna Cadoret and Arno Kret.
\newblock {Galois generic points on Shimura varieties}.
\newblock {Algebra and Number Theory Vol. 10, 1893--1934, (2016)}.

\bibitem{ct}
Anna Cadoret and Akio Tamagawa.
\newblock {On the geometric image of $\Ff_\ell$-linear representations of \'etale
fundamental groups}.
\newblock {Preprint}.


%
%
%
\bibitem{deligne-weil1}
Pierre Deligne. 
\newblock {La conjecture de Weil I}.
\newblock {\em {Publ. Math. IHES}}, 43:273--307 (1974).

\bibitem{deligne-weil2}
Pierre Deligne. 
\newblock {La conjecture de Weil II}.
\newblock {\em {Publ. Math. IHES}}, 52:137--252 (1980).

\bibitem{dixon}
Lance Dixon, Marcus du Sautoy, Avinoam Mann, Dan Segal.
\newblock {Analytic pro-$p$ groups}.
\newblock {\em {Cambridge studies in advanced mathematics}} (1999).

\bibitem{Drinfeld}
Vladimir Drinfeld, 
\newblock {On a conjecture of Deligne}.
\newblock{{\em Mosc. Math. J.} 12(3): 515�542, 668 (2012).}

%







\bibitem{huilarsen}
Chun-Yin Hui and Michael Larsen.
\newblock {\em Adelic openness without the Mumford-Tate conjecture.}
\newblock {Preprint.}


%
\bibitem{lp1}
Michael Larsen and Richard Pink.
\newblock {\em On $\ell$-independence of algebraic monodromy groups in compatible systems 
of representations.}
\newblock {Invent. Math.}, 107: 603--636 (1992).

\bibitem{lp}
Michael Larsen and Richard Pink.
\newblock {\em Abelian varieties, $\ell$-adic representations and $\ell$-independence.}
\newblock {Math.\ Ann.\ 302(3): 561--579 (1995)}.

\bibitem{pink}
Richard Pink.
\newblock {\em A Combination of the conjectures of Mordell-Lang and Andr\'e-Oort.}
\newblock {Math. Ann.}, 302(3): 561--579 (1995).


\bibitem{milne}
James Milne.
\newblock {\em {\'Etale Cohomology}}.
\newblock {Princeton University Press}, 1980.

\bibitem{milne-iag}
James Milne.
\newblock {\em {Algebraic Groups -- The theory of group schemes of finite type over a field}}.
\newblock {Lecture notes available at} {\sc www.jmilne.org}

\bibitem{serre67}
Jean-Pierre Serre.
\newblock {Sur les groupes de Galois attach\'es aux groups $p$-divisible.}
\newblock {In: Proceedings on a conference on local fields}, Springer (1967)

\bibitem{serretoribet2}
Jean-Pierre Serre.
\newblock {Lettre \`a Ken Ribet du 7/3/1986}.
\newblock {\em Collected Papers IV}.

\bibitem{serretoribet1}
Jean-Pierre Serre.
\newblock {Lettre \`a Ken Ribet du 1/1/1981 et du 29/1/1981}.
\newblock {\em Collected Papers IV}.

\bibitem{Serre-MW}
Jean-Pierre Serre.
\newblock {Lectures on the Mordell-Weil theorem. }
\newblock {Aspects of Mathematics, E15}, Viehweg, Braunschweig (1989)

\end{thebibliography}
\end{document}